\documentclass[14pt]{article}

\usepackage[english]{babel}
\usepackage[utf8]{inputenc}
\usepackage[T1]{fontenc}

\usepackage{dsfont}

\usepackage[a4paper,top=3cm,bottom=2cm,left=3cm,right=3cm,marginparwidth=1.75cm]{geometry}

\usepackage{amsmath, amsthm, amsfonts,accents}
 
\usepackage{graphicx}
\usepackage[colorinlistoftodos]{todonotes}
\usepackage[colorlinks=true, allcolors=blue]{hyperref}

\usepackage{mathpazo}

\usepackage{tikz}

\newtheorem{teo}{Theorem}
\newtheorem{prop}{Proposition}

\newtheorem{com}{Comentary}

\newcommand{\E}{\mathbb{E}}

\newcommand{\p}{\mathbb{P}}

\title{A note about an upper   bound for a   hitting time  of the fractional Ornstein-Uhlenbeck process}
\author{Wilson A. Cabanillas B. \\
 \small UFRJ - Instituto de Matem\'atica\\
  \small  wilson@dme.ufrj.br  \\
  \date{}
}

\begin{document}

\maketitle

\begin{abstract}
In this brief note we give an upper bound for $\p(\tau_u < T)$ with  $T>0$, where $\tau_u$ is  the hitting time defined as  $\tau_u:=\inf \{ t\geq 0 \, : \,  X_t\geq u \}$  and $(X_t)_{t\geq 0}$ is the fractional Ornstein-Uhlenbeck  processes which satisfies  the following stochastic differential equation 
     \begin{equation*}
   dX_t =-\lambda X_t dt+ \epsilon dB_t^H\quad \epsilon>0,\;\lambda>0
  \end{equation*}
  with $(B_t^H)_{t\geq 0}$ as the fractional brownian motion with parameter of Hurst $H\in ]0,1]$.
\end{abstract}

\section{Introduction}

%Upper and lower bounds for the expected value of the fractional Brownian motion are known, see for instance [\ref{Debicki}], but until now bounds for the  the expected value of the fractional  Ornstein Uhlenbeck process are unknown. 

 Let $(B_t^H)_{t\geq 0}$ be the fractional Brownian motion which is a centered Gaussian process with covariance function given by:
 $$\E(B_t B_s):=\frac{1}{2}\left( t^{2H}+s^{2H}-|t-s|^{2H} \right)$$
 The parameter $H \in \,]0,1]$ is called the parameter of Hurst. When $H=1/2$ we have that $B^H_t$ is the classical Brownian motion. Now we consider the following stochastic differential equation perturbed by the fractional white noise $B^H_t$:
 
  \begin{equation*}
   dX_t =-\lambda X_t + \epsilon dB_t^H\quad \epsilon>0,\;\lambda>0
  \end{equation*}
whose unique continuous solution given in [\ref{Cheridito}], with initial condition $x$, is

$$X^\epsilon_t=e^{-\lambda t} \left(x+\epsilon \int_0^t e^{\lambda s}dB_s^H\right).$$

This process $X_t^\epsilon$ is called the fractional Ornstein-Uhlenbeck process (fOUp). We find an upper bound  for the probability of the event  when the fOUp is above certain value $u>0$ by the first time before the fixed time $T>0$.  So we define 
$$\tau_u:=\inf \{ t\geq 0 \, : \,  X^\epsilon_t\geq u \}$$ 
 then we will find an upper estimate for $\p(\tau_u < T).$
 But  because of  the following equality of events 
  $$(\tau_u < T) =\big( \sup_{t\in [0,T]} X^\epsilon_t>u \big)$$  we will work with the last one. And finally using the Borell inequality \eqref{des-Bor}, we will get an upper bound for  $\p(\tau_u < T)=\p\left( \sup_{t\in [0,T]} X^\epsilon_t>u \right)$. 

In what follows, we will consider $\lambda=1$ and $x=0$. 
 
 \section{An upper bound }
 
 Here we show upper estimates for  $\p(\tau_u < T)=\p\left( \sup_{t\in [0,T]} X^\epsilon_t>u \right)$. The first theorem only works for $H\in \,[1/2,1]$, while the second works for all $H\in \,]0,1].$ 
%  The reason of this is given by Proposition \ref{est_e_sup}.

\begin{teo}\label{teo_prin} $\forall\, u> \E \left( \sup_{t\in [0,T]} X_t \right)$ and  {\color{blue}{ $H \in \;[1/2,1]$}}

\begin{equation}
\p\left( \sup_{t\in [0,T]} X^\epsilon_t>u \right)
  \leq 
 \exp \left\{ -\frac{u^2}{\epsilon^2}a_H+\frac{u}{\epsilon}b_H -c_H\right\} 
\end{equation} 
  where   $a_H:=\frac{1}{2T^{2H+1}} , $\;  \; $b_{H}:= \sqrt{\frac{2}{\pi}} \frac{  ({\color{red}2 }(H+1)+T) }{ (H+1)T^{H+1}}  $
and  $c_H= \frac{1}{\pi T}\cdot\left({\color{red}2 }+\frac{T}{H+1} \right)^2$
 
\end{teo}

 Using the Sudakov-Dudley inequality or Pisier theorem (see Appendix) we can
 extend  the preceeding theorem   for all $H \in \, ]0,1]$. This is the content of the following theorem.

\begin{teo}\label{teo_prin_2} $\forall\, u> \E \left( \sup_{t\in [0,T]} X_t \right)$ and {\color{blue}{ $H \in \;]0,1]$}}

\begin{equation}
\p\left( \sup_{t\in [0,T]} X^\epsilon_t>u \right)
  \leq 
 \exp \left\{ -\frac{u^2}{\epsilon^2}\bar a_H+\frac{u}{\epsilon}\bar b_H -\bar c_H\right\} 
\end{equation} 
 where 
 $\bar{a}_H=\frac{1}{2T^{2H+1}} , $\;  \; $\bar b_{H}:= \sqrt{\frac{2}{\pi}}\frac{ ({\color{red}2\pi\sqrt{2} }(H+1)+T)}{ (H+1)T^{H+1}}  $
and  $\bar c_H= \frac{1}{\pi T}\cdot\left({\color{red}2\pi\sqrt{2} }+\frac{T}{H+1} \right)^2$
 
\end{teo}

\begin{com}
The difference between these theorems is in the constants $b_H$ and $c_H$ and this difference is only due to the  estimates given in \eqref{est_Pisier} 
and \eqref{est_Debicki}.
\end{com}

The following proposition  will help us to prove the  Theorem \ref{teo_prin} and \ref{teo_prin_2}.

\begin{prop}\label{est_e_sup}

\begin{itemize}
\item[(a)]For  {\color{blue}{ $H \in \;[1/2,1]$}}   we have that
$$ \E \left(\sup_{t\in[0,T]} X^\epsilon_t \right)\leq \sqrt{\frac{2}{\pi}} T^H  \left[  2+ \frac{T}{H+1} \right]\cdot\epsilon $$ 
\item[(b)]The estimate obtained in (a) can be extended for all  {\color{blue}{ $H \in \;]0,1]$}}. We have that: 
$$ \E \left(\sup_{t\in[0,T]} X^\epsilon_t \right)\leq \sqrt{\frac{2}{\pi}} T^H  \left[ 2\pi \sqrt{2}+ \frac{T}{H+1} \right]\cdot\epsilon $$ 
\end{itemize}
 \end{prop}

\begin{proof}
For (a), as   $X^\epsilon_t=\epsilon  \cdot e^{-t} \int_0^t e^s dB_s^H$, then
\begin{eqnarray}\label{est_for_X}
\sup_{t\in[0,T]}X^\epsilon_t &\leq &\sup_{t\in[0,T]}|X^\epsilon_t|\leq \epsilon  \sup_{t\in[0,T]} e^{-t} \left|\int_0^t e^s dB_s^H \right|\notag\\
&=& \epsilon   \sup_{t\in[0,T]} e^{-t}
\left|e^t B_t^H -  \int_0^t e^u   B_u^H du   \right|\notag\\
&\leq & \epsilon   \sup_{t\in[0,T]} e^{-t}
\left(e^t|B_t^H|+  e^t\int_0^t    |B_u^H|du   \right)\notag\\
&\leq & \epsilon   \sup_{t\in[0,T]} 
  |B_t^H| \quad+\quad
  \epsilon\int_0^T   |B_u^H|du   
\end{eqnarray}
where in the first equality we have used integration by parts given in Proposition A.1 from [\ref{Cheridito}].\\

On the other hand, from   Lemma \ref{estimative_RfBm}   (b) (see Appendix) for $\gamma=1$ and $H \in[1/2,1]$ we get  that
\begin{eqnarray*}
\E\left(\sup_{t\in[0,T]} 
  |B_t^H| \right)&\leq &    2\sqrt{\frac{2}{\pi}} \,T^H
\end{eqnarray*}
Then,
\begin{eqnarray*}
\E\left(\sup_{t\in[0,T]}X^\epsilon_t \right)&\leq &
\epsilon   \cdot  2\sqrt{\frac{2}{\pi}} \,T^H +  \epsilon \cdot\int_0^T   \E|B_u^H|du  \\
&=&\epsilon   \cdot 2\sqrt{\frac{2}{\pi}} \,T^H +  \epsilon  \cdot\sqrt{\frac{2}{\pi}} \frac{T^{H+1}}{H+1}\\
&=&\epsilon \sqrt{\frac{2}{\pi}}T^H\left( 2+\frac{T}{H+1} \right)
\end{eqnarray*}

Finally we prove (b). From (a), taking expectation in both sides of \eqref{est_for_X}, we have that
$$\E\left(\sup_{t\in[0,T]}X^\epsilon_t \right)\leq \epsilon \cdot  \E\left( \sup_{t\in[0,T]} 
  |B_t^H| \right)+  \epsilon  \cdot\sqrt{\frac{2}{\pi}} \frac{T^{H+1}}{H+1} $$
  
And this time  to bound the second expectation we use what we got in \eqref{est_Pisier}. Therefore,
 $$\E\left(\sup_{t\in[0,T]}X^\epsilon_t \right)\leq \epsilon 4 \sqrt{\pi} T^H +  \epsilon  \cdot\sqrt{\frac{2}{\pi}} \frac{T^{H+1}}{H+1}=  \sqrt{\frac{2}{\pi}} T^H  \left[ 2\pi \sqrt{2}+ \frac{T}{H+1} \right]\cdot\epsilon$$   
\end{proof}

%%%%%%%%%%%%%%%%%%%%%%%%
%%%%%%%%%%%%%%%%%%%%%%%%
%%%%%%%%%%%%%%%%%%%%%%%%
The next proposition  is also an ingredient    which we will put into the Borell inequality  in order to estimate $\p\left( \sup_{t\in [0,T]} X^\epsilon_t>u \right)$,
\begin{prop}
We have that,
\begin{equation*}
\E (X^\epsilon_t)^2 =  \epsilon^2 \left[ t^{2H}e^{-t}+\tfrac{1}{2}\int_0^t u^{2H}(e^{-u}-e^{u-2t}) du   \right]
\end{equation*}
Therefore, for $  H\in \left]0,1\right]\;$

\begin{equation*}
\sigma_T^2=\sup_{t\in [0,T]} \E (X^\epsilon_t)^2  \leq \epsilon^2 T^{2H+1}
\end{equation*}

\end{prop}

\begin{proof}

We know that  $X^\epsilon_t=\epsilon  \cdot e^{-t} \int_0^t e^s dB_s^H=\epsilon \cdot e^{-t} \left( e^t B_t^H -  \int_0^t e^u   B_u^H du\right)$, then 

\begin{equation}\label{est_4_m19}
\E (X^\epsilon_t)^2= \epsilon^2\cdot t^{2H}-2\epsilon^2e^{-t} \E \left( B_t^H \int_0^tB_u^He^u du \right)+\epsilon^2 \cdot e^{-2t} \E\left(\int_0^tB_u^He^u du \right)^2
\end{equation}

%%%%%%%%%%%%%%%%%%%%%%%%%%%%%%%%%%%%%%%
%%%%%%%%%%%%%%%%%%%%%%%%%%%%%%%%%%%%%%%%

Now, let's compute $ \E \left( B_t^H \int_0^tB_u^He^u du \right)$,\,
 and   $  \E\left(\int_0^tB_u^He^u du \right)^2$ one by one. Indeed,

\begin{eqnarray}\label{eq3_26_4}
 \E\left(B_t^H\int_0^tB_u^He^u du \right) &=& \int_0^t \E   (B_t^H  B_u^H)e^u du = \int_0^t  \tfrac{1}{2}(t^{2H}+u^{2H}-|t-u|^{2H})e^u du \notag\\
% &=&  \int_0^t  \tfrac{1}{2}(t^{2H}+u^{2H}-(t-u)^{2H})e^u du\\
 &=&  \int_0^t  \tfrac{1}{2}t^{2H}e^u du+ \int_0^t  \tfrac{1}{2}u^{2H}e^u du -
   \int_0^t  \tfrac{1}{2}(t-u)^{2H}e^u du\notag\\
   &\overset{v=t-u}{=} & \tfrac{1}{2} t^{2H}(e^t-1)+\int_0^t  \tfrac{1}{2}u^{2H}e^u du -  \int_0^t  \tfrac{1}{2}v^{2H}e^{t-v} dv\notag\\
   &=&\tfrac{1}{2}t^{2H}(e^t-1)+\int_0^t  \tfrac{1}{2}u^{2H}e^u du - e^t \int_0^t  \tfrac{1}{2}u^{2H}e^{-u} du
\end{eqnarray}

\begin{eqnarray}\label{eq1_jue_25_04}
   \E\left(\int_0^tB_u^He^u du \right)^2   & = &    \E\left(\int_0^tB_u^He^u du  \int_0^tB_s^He^s ds\right)\notag \\
    &= &  \int_0^t \int_0^t \E\left( B_u^H  B_s^H  \right)  e^{u+s} ds  du \notag\\
   & =&  \frac{1}{2} \int_0^t \int_0^t  \left( u^{2H}+s^{ 2H}-|u-s|^{2H} \right)  e^{u+s} ds  du \;\;\;\;
\end{eqnarray}
 
The double integral may be written as
 
 \begin{equation*}
2  \int_0^t \int_0^u \left( u^{2H}+s^{ 2H}-(u-s)^{2H} \right)  e^{u+s} ds  du
\end{equation*}
    then in \eqref{eq1_jue_25_04}

\begin{equation*}
 \E\left(\int_0^tB_u^He^u du \right)^2 = \int_0^t \int_0^u \left( u^{2H}+s^{ 2H}-(u-s)^{2H} \right)  e^{u+s} ds  du
\end{equation*}
This last double integral we split up into the sum  $I_1+I_2 -I_3 $, where

\begin{eqnarray*}
 I_1 &=&\int_0^t \int_0^u   u^{2H} e^{u+s} ds  du  \\
 I_2 &=& \int_0^t \int_0^u  s^{2H} e^{u+s} ds  du \\
 I_3&=&  \int_0^t \int_0^u   (u-s)^{2H}    e^{u+s} ds  du
\end{eqnarray*}
For $I_1$ we have that,
$$I_1=\int_0^t u^{2H}e^{2u}du- \int_0^t u^{2H}e^{u}du$$
For $I_2$ we get that,

\begin{eqnarray*}
I_2=\int_0^t \int_0^u  s^{2H} e^{u+s} ds  du &=& 
 \int_0^t \int_s^t  s^{2H} e^{u+s} du  ds \\
 &= &  \int_0^t  s^{2H} e^{s} \left(\int_s^t e^u du \right)  ds  \\
 &=& \int_0^t s^{2H}e^{s+t}ds -  \int_0^t s^{2H}e^{2s}ds\\
 &=& e^t \cdot\int_0^t u^{2H}e^{u}du -  \int_0^t u^{2H}e^{2u}du
\end{eqnarray*}
 
For $I_3$, doing a variable change $v=u-s$ and then integrating by parts, we have that

\begin{eqnarray*}
 I_3 &=& \int_0^t \int_0^u   (u-s)^{2H}    e^{u+s} ds  du= \int_0^t e^u \left( 	\int_0^u  v^{2H}e^{u-v} dv \right)du\\
 &=&  \int_0^t e^{2u} \left( 	\int_0^u  v^{2H}e^{-v} dv \right)du\\
 &=& \left(\tfrac{1}{2}e^{2u} \cdot\int_0^u v^{2H}e^{-v}dv \right) \Bigg|_{u=0}^{u=t} - \int_0^t   \tfrac{1}{2}e^{2u} u^{2H}e^{-u}du \\
 &= &  \tfrac{1}{2}e^{2t} \cdot\int_0^t v^{2H}e^{-v}dv  - \int_0^t   \tfrac{1}{2}e^{u} u^{2H} du\\
  &= &  \tfrac{1}{2}e^{2t} \cdot\int_0^t u^{2H}e^{-u}du  - \int_0^t   \tfrac{1}{2}u^{2H}e^{u}  du
\end{eqnarray*}

Thus, 

\begin{eqnarray}\label{eq2_26_04}
\E\left(\int_0^tB_u^He^u du \right)^2& =&
  e^t \cdot\int_0^t u^{2H}e^{u}du -\int_0^t u^{2H}e^{u}du \notag\\
  &  -&  \tfrac{1}{2}e^{2t} \cdot\int_0^t u^{2H}e^{-u}du  +    \int_0^t   \tfrac{1}{2}u^{2H}e^{u}  du\notag\\
  &=& \left(e^t-\frac{1}{2} \right)\int_0^t u^{2H}e^{u}du- \tfrac{1}{2}e^{2t} \cdot\int_0^t u^{2H}e^{-u}du \qquad \quad
 \end{eqnarray}

Replacing  \eqref{eq3_26_4} and \eqref{eq2_26_04}  in \eqref{est_4_m19} and  simplifying we have that

\begin{eqnarray*}\label{eq1_01_05}
\E (X^\epsilon_t)^2&=&\epsilon^2 \left[ t^{2H}-2 e^{-t} \E \left( B_t^H \int_0^tB_u^He^u du \right)+  e^{-2t} \E\left(\int_0^tB_u^He^u du \right)^2 \right]\notag\\
 &= & \epsilon^2 \left[ t^{2H}e^{-t}+\tfrac{1}{2}\int_0^t u^{2H}(e^{-u}-e^{u-2t}) du   \right]
\end{eqnarray*}

Let $\phi(u)=e^{-u}-e^{u-2t},\; u\in[0,t]$, because $\phi'(u)<0$  then $\phi$ is decreasing and we have  %in  
\begin{eqnarray}\label{e2_01_05}
\E (X^\epsilon_t)^2&\leq &  \epsilon^2 \left[ t^{2H}e^{-t}+\tfrac{1}{2} (1-e^{-2t})\frac{t^{2H+1}}{2H+1}   \right]
\end{eqnarray}

Therefore, 
\begin{eqnarray*}
\sigma_T^2=\sup_{t\in [0,T]} \E (X^\epsilon_t)^2 &\leq &
 \epsilon^2 \sup_{t\in [0,T]}  \left[ t^{2H}e^{-t}+\tfrac{1}{2} \underbrace{(1-e^{-2t})}_{\leq 1}\frac{t^{2H+1}}{2H+1}   \right]\\
  &\leq & \epsilon^2 \left[ T^{2H}+\frac{1}{2}\cdot \frac{T^{2H+1}}{2H+1} \right]\\
  & \leq & \epsilon^2 T^{2H+1}
    \end{eqnarray*} 
\end{proof}

%%%%%%%%%%%%%%%%%%%%%%%%%%%%%%%%%%%%%%%
%%%%%%%%%%%%%%%%%%%%%%%%%%%%%%%%%%%%%%%

 \section{Proofs for the upper bounds }

 \begin{center}
\textbf{ Proof of Theorem 1}
 \end{center} 
\begin{proof}
Using the two propositions above together with the Borell inequality, we obtain that for
\begin{center}
 $u\geq \sqrt{\frac{2}{\pi}} T^H  \left[  2+ \frac{T}{H+1} \right]\cdot\epsilon$ 
 
\end{center}

\begin{eqnarray*}
\p\left( \sup_{t\in [0,T]} X^\epsilon_t>u \right)& =  &   \exp \left\{  \frac{-(u-\E \sup_{t\in [0,T]} X^\epsilon_t)^2}{2\sigma_T^2} \right\}  \\
&  \leq &  \exp \left\{  \frac{-(u- \sqrt{\frac{2}{\pi}} T^H  \left[  2+ \frac{T}{H+1} \right]\cdot\epsilon )^2}{2 \epsilon^2  T^{2H+1} } \right\}  \\
  & =& \exp \left\{ -a_H\left(\frac{u}{\epsilon}\right)^2+b_H\frac{u}{\epsilon}-c_H \right\}  
\end{eqnarray*}
 where $a_H:=\frac{1}{2T^{2H+1}} , $\;  \; $b_{H}:= \frac{\sqrt{2}(2(H+1)+T)}{\sqrt{\pi}(H+1)T^{H+1}}  $
and  $c_H= \frac{1}{\pi T}\cdot\left(2+\frac{T}{H+1} \right)^2$

 \end{proof}

 \begin{center}
\textbf{ Proof of Theorem 2}
 \end{center} 

\begin{proof}
We proceed almost exactly as in the proof of Theorem \ref{teo_prin}. The difference is that here we consider 
$$u> \sqrt{\frac{2}{\pi}} T^H  \left[ 2\pi \sqrt{2}+ \frac{T}{H+1} \right]\cdot\epsilon$$
And then,  as before, by the Borell inequality, we have that

\begin{eqnarray*}
  \p\left( \sup_{t\in [0,T]} X^\epsilon_t>u \right) 
&  \leq &  \exp \left\{  \frac{-(u- \sqrt{\frac{2}{\pi}} T^H  \left[  2\pi \sqrt{2}+ \frac{T}{H+1} \right]\cdot\epsilon )^2}{2  \epsilon^2  T^{2H+1} } \right\}  \\
  & =& \exp \left\{ -\bar a_H\left(\frac{u}{\epsilon}\right)^2+\bar b_H\frac{u}{\epsilon}-\bar c_H \right\}  
\end{eqnarray*}

 where $\bar a_H:=\frac{1}{2T^{2H+1}} , $\;   $\bar b_{H}:= \frac{\sqrt{2}(2\pi \sqrt{2}(H+1)+T)}{\sqrt{\pi}(H+1)T^{H+1}}  $
and  $\bar c_H= \frac{1}{\pi T}\cdot\left(2\pi \sqrt{2}+\frac{T}{H+1} \right)^2$

\end{proof}

\section{Appendix}

Let  $(S,d)$ be a metric space. 
For $\epsilon>0$ we define $N(\epsilon)$ as the minimal number of sets in a cover of $S$ by subsets of $d$-diameter   not exceeding    $\epsilon$.

Now, let   $(X_t)_{t\in S}$ be a Gaussian process indexed by   $S$. Let $d$  be the (pseudo)metric  defined as follows
$$d(s,t):=\sqrt{\E(X_t-X_s)^2}$$

\begin{teo}\label{D_S_D}(Sudakov - Dudley inequality)
For a centered	 Gaussian  process   $(X_t)_{t\in S}$  
$$E\left(\sup_{t\in S}X_t \right)\leq  4\sqrt{2} \int_0^{ \sigma/2 }\left(\log N(\epsilon) \right)^{1/2}d\epsilon$$
where $\sigma^2=\sup_{t\in S} \E X_t^2.$
\end{teo}

\begin{proof}
See [\ref{Lifshits}], Theorem 10.1.
\end{proof}

And, in general, for a centered stochastic process, not necessarily Gaussian we have the following inequality:

\begin{teo}\label{Pisier}(Pisier theorem) 
For a centered stochastic process  $(X_t)_{t\in S}$  (not necessarily Gaussian)  that satisfies $\sigma^2=\sup_{t\in S} \E X_t^2<+\infty$   
$$E\left(\sup_{t\in S}X_t \right)\leq  4 \int_0^{ \sigma  }\left(\log N(\epsilon) \right)^{1/2}d\epsilon$$
where $\sigma^2=\sup_{t\in S} \E X_t^2.$
\end{teo}

\begin{proof}
See [\ref{Lifshits}], Exercise 10.1.
\end{proof}

In our case, $S=[0,T]$ and $(X_t)_{t\in S}$ is the fBm $(B_t^H)_{t\in [0,T]}$ with Hurst parameter $H\in \,]0,1]$. Then, due to stationarity of fBm, it follows that
$$d(s,t)=\sqrt{\E(B_t^H-B_s^H)^2}=\sqrt{\E(B^H_{t-s})^2}$$
 with $\E(B^H_{t-s})^2=(t-s)^{2H}$, then
 $$d(s,t)=|t-s|^H$$
 
 On the other hand, we can find in a explicit manner   $\sigma^2$. Indeed, as $\E (B^H_t)^2=t^{2H}$, then 
 $$\sigma^2=\sup_{t\in[0,T]}\E (B^H_t)^2=T^{2H}$$
 
To find, $N(\epsilon)$, note that the intervals of  $d$-diameter $\epsilon$ can be of the form: $[0,\epsilon^{1/H}], [\epsilon^{1/H},2\epsilon^{1/H}]$, etc. Then, if we want to cover the interval $[0,T]$ with $N(\epsilon)$ intervals of $d$-diameter $\epsilon$ we must have that
$$N(\epsilon)\cdot d(\epsilon^{1/H},0)=d(T,0)$$ 
 then, $$N(\epsilon)=\frac{T^H}{\epsilon}$$
Using Theorem [\ref{D_S_D}], it follows that
 
\begin{eqnarray*}
\E\left(\sup_{t\in [0,T]}B^H_t \right)&\leq & 4\sqrt{2} \int_0^{T^H/2 }\left(\log \frac{T^H}{\epsilon} \right)^{1/2}d\epsilon
\end{eqnarray*}
and  making a change of variable $\epsilon=T^H x$,  we have that

$$\int_0^{T^H/2 }\left(\log \frac{T^H}{\epsilon} \right)^{1/2}d\epsilon=T^H\int_0^{1/2 }\left(\log \frac{1}{x} \right)^{1/2}dx\approx   0.628114  \cdot T^H $$

Therefore, we get the following estimate (for all $H\in\,]0,1]$)
\begin{eqnarray*}
\E\left(\sup_{t\in [0,T]}B^H_t \right) \leq  4\sqrt{2}  T^H\int_0^{1/2 }\left(\log \frac{1}{x} \right)^{1/2}dx &\approx& 4\sqrt{2} \times 0.628114 \cdot T^H \\
& \approx &3.55315\cdot T^H
\end{eqnarray*}

On the other hand, using Theorem [\ref{Pisier}], we get a little better estimate (for all $H\in\,]0,1]$):

\begin{eqnarray}\label{est_Pisier}
\E\left(\sup_{t\in [0,T]}B^H_t \right) \leq  4  T^H\int_0^{1 }\left(\log \frac{1}{x} \right)^{1/2}dx &=& 4T^H \frac{\sqrt{\pi}}{2}  \notag \\
& \approx &3.544908\cdot T^H\qquad
\end{eqnarray}
both of which are \textcolor{red}{\textbf{not better}} than the estimate given by
   Debicki, K.  and  Tomanek, A. in [\ref{Debicki}] who obtained (only for $H\in \,[1/2,1]$)
\begin{equation}\label{est_Debicki}
\E\left(\sup_{t\in [0,T]}B^H_t \right) \leq \sqrt{\frac{2}{\pi}} \cdot T^H \approx 0.797885 \cdot T^H
\end{equation}

\begin{teo}(Borell inequality)
Let $(X_t)$ be a  centered  Gaussian process  a.s bounded on $[0,T]$. Then $$\E \left( \sup_{t\in [0,T]} X_t \right)<\infty$$
and     $\forall\, u> \E \left( \sup_{t\in [0,T]} X_t \right)$

\begin{equation}\label{des-Bor}
 \p\left( \sup_{t\in [0,T]} X_t>u \right)
 \leq \exp \left\{ -\frac{(u-\E \sup_{t\in[0,T]} X_t)^2}{2\sigma^2_T} \right \} 
\end{equation}
where $\sigma^2_T= \sup_{t\in[0,T]}\E X^2_t$.
\end{teo}

\begin{proof}
See [\ref{Adler}], Teorema 2.1 or  [\ref{Adler+Taylor}], Teorema 2.1.1.
\end{proof}

\begin{teo}( Slepian inequality)
If $X$ and $Y$ are  centered  Gaussian process a.s bounded on $[0,T]$ such that $\E X_t^2=\E Y_t^2$ for all   $t\in [0,T]$ and 
$$\E(X_t-X_s)^2\leq \E(Y_t-Y_s)^2\qquad \forall\, 	s,t\in[0,T], $$
then for all real number   $\lambda$
$$\p\left( \sup_{t\in[0,T]}X_t>\lambda \right)\leq \p\left( \sup_{t\in[0,T]}Y_t>\lambda \right)$$
\end{teo}

The following theorem (which is   Theorem 1.1 from [\ref{Debicki}])  ), will allow us to find estimates for every moment of the supremum of the reflected fBm under certain hypothesis about the variance function.

\begin{teo}\label{estimative_RfBm}
Let  $(X(t))_{t\geq 0}$ be a centered Gaussian process such that $X(0)=0$ a.s with stationary increments and variance function $\sigma_X^2$ that is continuous and strictly increasing.
\begin{enumerate}
\item[(a)] If  $\sigma_X^2$  is  sub-additive, in the  following sense: $\forall\;\; 0 \leq s \leq t \leq T$
$$ \sigma_X^2(t)\leq \sigma_X^2(t-s) + \sigma_X^2(s)$$
then, 
$$\E\left[\sup_{t\in [0,T]} X(t) \right]^\gamma \geq \left(\sigma_X^2(T)\right)^{\gamma/2} \tfrac{1}{\sqrt{\pi}}2^{\gamma/2}\Gamma\left( \frac{\gamma+1}{2} \right)$$
\item[(b)] If $\sigma_X^2$  is super-additive,in the following sense $\forall\;\; 0 \leq s \leq t \leq T$
$$ \sigma_X^2(t)\geq \sigma_X^2(t-s) + \sigma_X^2(s)$$
then, 
\begin{equation*}\label{estimative_RfBm}
\E\left[\sup_{t\in [0,T]} X(t) \right]^\gamma \leq \left(\sigma_X^2(T)\right)^{\gamma/2} \tfrac{1}{\sqrt{\pi}}2^{\gamma/2}\Gamma\left( \frac{\gamma+1}{2} \right)
\end{equation*}

\end{enumerate}

\end{teo}

\end{document}